\documentclass[11pt]{article}

\usepackage{mathrsfs}
\usepackage{graphicx,amssymb,mathrsfs,amsmath,latexsym,amsfonts,titlesec,amscd,epsfig,cite,color,
                                                                                   }

\def\R{{\mathbb{R}}}
\def\BMO{{\mathrm{BMO}}}
\def\loc{{\mathrm{loc}}}
\DeclareMathOperator*{\essinf}{ess\, inf}
\DeclareMathOperator*{\esssup}{ess\, sup}

\newcommand{\upcite}[1]{\textsuperscript{\textsuperscript{\cite{#1}}}}
\numberwithin{equation}{section} 
\allowdisplaybreaks   
\usepackage{indentfirst} 
\usepackage[amsmath,thmmarks]{ntheorem} 

\theoremstyle{definition} 
\theoremheaderfont{\normalfont\rmfamily\bf}
\theorembodyfont{\normalfont\rm } \theoremindent0em
\theoremseparator{\hspace{-.5em}}
\newtheorem{theorem}{\indent
                  Theorem}[section]
    \newtheorem{lemma}{\indent  Lemma} [section]

\newtheorem{proposition}{\indent  Proposition}[section]

    \newtheorem{definition}{\indent  Definition} [section]
        \newtheorem{remark}{\indent  Remark}  

\newtheorem*{acknowledgments}{\indent Acknowledgments\quad }
\theoremstyle{nonumberplain} 
\theoremheaderfont{\bf\rmfamily} \theorembodyfont{\normalfont \rm }
\theoremsymbol{$\blacksquare$}
\newtheorem{proof}{\indent Proof}

\title{\bf\Large
Boundedness  for fractional Hardy-type operator  on variable exponent Herz-Morrey spaces }
\author{\small \textsc{Jiang-Long Wu,  Wen-Jiao Zhao } 
\\
 {\footnotesize\it Faculty of Information Technology, Macau University of Science and Technology, Macau,  China}\\
{\footnotesize\it Department of Mathematics, Mudanjiang Normal University, Mudanjiang, 157011, China}
  }
\date{} 
\begin{document}

\maketitle

\footnote{ 
\textit{AMS} (2010) \textit{Mathematics Subject Classification}:
Primary 42B20; Secondary 47B38. }

\begin{abstract}
In this paper, the fractional Hardy-type operator of variable order $\beta(x)$ is shown to be bounded
from the variable exponent  Herz-Morrey spaces $M\dot{K}_{p_{_{1}},q_{_{1}}(\cdot)}^{\alpha(\cdot),\lambda}(\R^{n})$ into the weighted space $M\dot{K}_{p_{_{2}},q_{_{2}}(\cdot)}^{\alpha(\cdot),\lambda}(\R^{n},\omega)$, where $\alpha(x)\in L^{\infty}(\mathbb{R}^{n})$ be log-H\"{o}lder continuous both at the origin and at infinity, $\omega=(1+|x|)^{-\gamma(x)}$ with some $\gamma(x)>0$ and $ 1/q_{_{1}}(x)-1/q_{_{2}}(x)=\beta(x)/n$ when $q_{_{1}}(x)$ is not necessarily constant at infinity.
\end{abstract}

\section{Introduction}

Let $f$ be a locally integrable function on $\R^{n}$. The $n$-dimensional Hardy operator is defined by
$$
\mathscr{H}(f)(x):= \frac{1}{|x|^{n}} \int_{|t|<|x|} f(t) \mathrm{d}t,\ \  \ x\in \R^{n}\setminus \{0\}.
$$

In 1995, Christ and Grafakos\upcite{CG} obtained the result for the
boundedness of $\mathscr{H}$ on $L^{p}(\R^{n})\ (1<p<\infty)$
spaces, and they also found the exact operator norms of $\mathscr{H}$ on
this space. In 2007, Fu et al\upcite{FLLW} gave the central $\BMO$ estimates for
commutators of  $n$-dimensional fractional and Hardy  operators. And recently, the first author\upcite{WW,W,WL,WZ,WZ1,ZW,W1} also considers the boundedness for Hardy  operator and its commutator in (variable exponent) Herz-Morrey spaces.

Nowadays there is an evident increase of investigations related to both the theory of the variable exponent function spaces   and the operator theory in these spaces. This is caused by possible applications to models with non-standard local growth (in elasticity theory, fluid mechanics, differential equations and image processing, see for example \cite{R,DR,CLR,HHLT,SS} and references therein) and is based on the breakthrough result on boundedness of the Hardy-Littlewood maximal operator in these spaces
 (for more details see \cite{KR,CDF,CFMP,CFN,D1,D2,DHHMS,K,L,N,PR} et al).

 We first define the $n$-dimensional  fractional Hardy-type operators with variable order $\beta(x)$  as follows.

\begin{definition}\ \   \label{def.1.1}
Let $f$ be a locally integrable function on $\R^{n},~ 0\le\beta(x)<n$. The $n$-dimensional  fractional Hardy-type operators of variable order $\beta(x)$ are defined by
 \begin{alignat}{2}
 \mathscr{H}_{\beta(\cdot)}(f)(x) &:= \frac{1}{|x|^{n-\beta(x)}} \int_{|t|<|x|} f(t)\mathrm{d}t,  \tag{\theequation a}  \label{equ.1.a}\\
\mathscr{H}^{\ast}_{\beta(\cdot)}(f)(x) &:=\int_{|t|\ge|x|} \frac{f(t)}{|t|^{n-\beta(x)}} \mathrm{d}t,  \tag{\theequation b}   \label{equ.1.b}
\end{alignat}
where $x\in \R^{n}\setminus \{0\}$.
\end{definition}

Obviously, when $\beta(x)=0$, $\mathscr{H}_{\beta(\cdot)}$ is just $\mathscr{H}$, and denote by $\mathscr{H}^{\ast}:=\mathscr{H}^{\ast}_{\beta(\cdot)}=\mathscr{H}^{\ast}_{0}$. And when $\beta(x)$ is constant, $\mathscr{H}_{\beta(\cdot)}$ and $\mathscr{H}^{\ast}_{\beta(\cdot)}$ will become $\mathscr{H}_{\beta}$ and $\mathscr{H}^{\ast}_{\beta}$\upcite{FLLW} respectively.

The Riesz-type potential operator with variable order $\beta(x)$ is defined by
\begin{equation} \label{equ.1.1}
I_{\beta(\cdot)}(f)(x)=\int_{\R^{n}} \frac{f(y)}{|x-y|^{n-\beta(x)}}  \mathrm{d}y,\  \ 0<\beta(x)<n.
\end{equation}

In 2004, Diening\upcite{D} proved Sobolev's theorem for the potential  $I_{\beta}$ on
the whole space $\R^{n}$ assuming that $p(x)$ is constant at infinity ($p(x)$ is always constant outside some large ball) and satisfies the same logarithmic condition as in \cite{S}. Another progress for unbounded domains is the  result of Cruz-Uribe et al\upcite{CFN} on the boundedness of the maximal operator in unbounded domains for exponents $p(x)$ satisfying the logarithmic smoothness condition both locally and at infinity.

In \cite{KS}, Kokilashvili and Samko prove Sobolev-type theorem for the potential $I_{\beta(\cdot)}$ from the space $L^{p(\cdot)}(\R^{n})$  into the weighted space $L^{q(\cdot)}_{\omega}(\R^{n})$  with the power weight $\omega$ fixed to infinity, under the logarithmic condition for  $p(x)$ satisfied locally and at infinity, not supposing that  $p(x)$ is constant at infinity but assuming that $p(x)$ takes its minimal value at infinity.

In addition, the theory of function spaces with variable exponent has rapidly made progress in the past twenty years
since some elementary properties were established by Kov\'{a}\v{c}ik-R\'{a}kosn\'{i}k\upcite{KR}.

In 2012,  Almeida and  Drihem\upcite{AD} discuss the boundedness of a wide class of sublinear operators on Herz spaces $K_{q(\cdot)}^{\alpha(\cdot),p}(\R^{n})$ and $\dot{K}_{q(\cdot)}^{\alpha(\cdot),p}(\R^{n})$ with variable exponent $\alpha(\cdot)$ and $q(\cdot)$.
 Meanwhile, they also established Hardy-Littlewood-Sobolev theorems for fractional integrals on variable Herz spaces.
  In 2013, Samko\upcite{S1,S2} introduced a new Herz type function spaces with variable exponent,  where all the three parameters are variable, and proved the boundedness of some sublinear operators (also ref. \cite{IN}).
In 2014, Izuki and Noi\upcite{IN1} concerned with duality and reflexivity of   Herz spaces $K_{q(\cdot)}^{\alpha(\cdot),p(\cdot)}(\R^{n})$ and $\dot{K}_{q(\cdot)}^{\alpha(\cdot),p(\cdot)}(\R^{n})$  with variable exponents.
Moreover,  in recently, Wu\upcite{W1} considers the boundedness for fractional Hardy-type operator on Herz-Morrey spaces $ M\dot{K}_{p, q(\cdot)}^{\alpha(\cdot), \lambda}(\mathbb{R}^{n})$  with variable exponent $q(\cdot)$ but fixed $ \alpha\in \mathbb{R}$ and $p\in(0,\infty)$.

 Motivated by the above results, we are to investigate mapping properties of the  fractional Hardy-type operators $\mathscr{H}_{\beta(\cdot)}$ and $\mathscr{H}^{\ast}_{\beta(\cdot)}$  within the framework of the variable exponent Herz-Morrey spaces $ M\dot{K}_{p, q(\cdot)}^{\alpha(\cdot), \lambda}(\mathbb{R}^{n})$.

Throughout this paper, we will denote by $|S|$ the Lebesgue measure  and  by $\chi_{_{\scriptstyle S}}$ the characteristic function  for a measurable set  $S\subset\R^{n}$; $B(x,r)$ is the ball cenetered at $x$ and of radius $r$;~$B_{0}=B(0,1)$. $C$  denotes a constant that is independent of the main parameters involved but
whose value may differ from line to line. For any index $1< q(x)< \infty$, we denote by $q'(x)$ its conjugate index,
namely, $q'(x)=\frac{q(x)}{q(x)-1}$.  For $A\sim D$, we mean that there is a constant $C > 0$ such that $C^{-1}D\le A \le CD$.

\section{Preliminaries}

In this section, we give the definition of Lebesgue and Herz-Morrey spaces with variable exponent, and give basic properties
and useful lemmas.

\subsection{Function spaces with variable exponent}

Let $\Omega$ be a measurable set in $\R^{n}$ with $|\Omega|>0$. We first define variable exponent Lebesgue spaces.

\begin{definition} \label{def.2.1}
\ \ Let ~$ q(\cdot): \Omega\to[1,\infty)$ be a measurable function.

 \begin{list}{}{} 
\item[(I)]  \ The Lebesgue spaces with variable exponent $L^{q(\cdot)}(\Omega)$ is defined by
  $$ L^{q(\cdot)}(\Omega)=\{f~ \mbox{is measurable function}:  F_{q}(f/\eta)<\infty ~\mbox{for some constant}~ \eta>0\}, $$
  where $F_{q}(f):=\int_{\Omega} |f(x)|^{q(x)} \mathrm{d}x$. The Lebesgue space $L^{q(\cdot)}(\Omega)$ is a Banach function space with respect to the norm
  \begin{equation*}
   \|f\|_{L^{q(\cdot)}(\Omega)}=\inf \Big\{ \eta>0:  F_{q}(f/\eta)=\int_{\Omega} \Big( \frac{|f(x)|}{\eta} \Big)^{q(x)} \mathrm{d}x \le 1 \Big\}.
\end{equation*}
 \item[(II)] \ The space $L_{\loc}^{q(\cdot)}(\Omega)$ is defined by
  $$ L_{\loc}^{q(\cdot)}(\Omega)=\{f ~\mbox{is measurable}: f\in L^{q(\cdot)}(\Omega_{0}) ~\mbox{for all compact  subsets}~ \Omega_{0}\subset \Omega\}.  $$

 \item[(III)] \ The weighted Lebesgue space $L_{\omega}^{q(\cdot)}(\Omega)$ is defined by as the set of all measurable functions for which $$\|f\|_{L^{q(\cdot)}_{\omega}(\Omega)}=\|\omega f\|_{L^{q(\cdot)}(\Omega)}<\infty.$$
\end{list}
\end{definition}

Next we define some classes of variable  exponent functions. Given a function $f\in L_{\loc}^{1}(\R^{n})$, the Hardy-Littlewood maximal operator $M$ is defined by
$$Mf(x)=\sup_{r>0} r^{-n} \int_{B(x,r)} |f(y)| \mathrm{d}y,$$
where $B(x,r)=\{y\in \R^{n}: |x-y|<r\}$.

\begin{definition} \label{Def.2.2}\ \
 \ Given a measurable function $q(\cdot)$ defined on $\R^{n}$, we write
$$q_{-}:=\essinf_{x\in \R^{n}} q(x),\ \ q_{+}:= \esssup_{x\in \R^{n}} q(x).$$

\begin{list}{}{}
\item[(I)]\   $q'_{-}=\essinf\limits_{x\in \R^{n}} q'(x)=\frac{q_{+}}{q_{+}-1},\ \ q'_{+}= \esssup\limits_{x\in \R^{n}} q'(x)=\frac{q_{-}}{q_{-}-1}.$

\item[(II)]\ Denote by $\mathscr{P}(\R^{n})$ the set of all measurable functions $ q(\cdot): \R^{n}\to(1,\infty)$ such that
$$1< q_{-}\le q(x) \le q_{+}<\infty,\ \ x\in \R^{n}.$$

\item[(III)]\  The set $\mathscr{B}(\R^{n})$ consists of all  measurable functions  $q(\cdot)\in\mathscr{P}(\R^{n})$ satisfying that the Hardy-Littlewood maximal operator $M$ is bounded on $L^{q(\cdot)}(\R^{n})$.

\end{list}

\end{definition}

\begin{definition} \label{Def.2.3}\ \
 \ Let $\alpha(\cdot)$ be a real-valued function on  $\R^{n}$.

\begin{list}{}{}
\item[(I)]\   The set $\mathscr{C}^{\log}_{loc}(\R^{n})$  consists of all local $\log$-H\"{o}lder continuous functions $ \alpha(\cdot)$   satisfies
 \begin{equation*}    
 |\alpha(x)-\alpha(y)| \le \frac{-C}{\ln(|x-y|)},  \ \ \  |x-y|\le 1/2,\ x,y \in \R^{n}.
\end{equation*}

\item[(II)]\  The set $\mathscr{C}^{\log}_{0}(\R^{n})$ consists of all  $\log$-H\"{o}lder continuous functions $ \alpha(\cdot)$ at origin   satisfies
    \begin{equation}  \label{equ.2.2}
 |\alpha(x)-\alpha(0)| \le \frac{C}{\ln(e+\frac{1}{|x|})},  \ \ \ \ x \in \R^{n}.
\end{equation}

\item[(III)]\  The set $\mathscr{C}^{\log}_{\infty}(\R^{n})$ consists of all  $\log$-H\"{o}lder continuous functions $ \alpha(\cdot)$ at infinity   satisfies
\begin{equation}\label{equ.2.3}
 |\alpha(x)-\alpha_{\infty}| \le \frac{C_{\infty}}{\ln(e+|x|)},  \ \ \ x \in \R^{n},
\end{equation}
 where $\alpha_{\infty}=\lim\limits_{|x|\to \infty}\alpha(x)$.

\item[(IV)]\  Denote by  $\mathscr{C}^{\log}(\R^{n}):=\mathscr{C}^{\log}_{loc}(\R^{n})\cap \mathscr{C}^{\log}_{\infty}(\R^{n})$ the set of all global $\log$-H\"{o}lder continuous functions $ \alpha(\cdot)$.

\end{list}

\end{definition}

\begin{remark}\quad The $\mathscr{C}^{\log}_{\infty}(\R^{n})$ condition is equivalent to the uniform continuity condition
\begin{equation*}  
  |q(x)-q(y)| \le \frac{C}{\ln(e+|x|)},  \ \ \  |y|\ge|x|,\ x,y \in \R^{n}.
\end{equation*}
The $\mathscr{C}^{\log}_{\infty}(\R^{n})$ condition was originally defined in this form in \cite{CFN}.
\end{remark}

  Next we define variable exponent Herz-Morrey spaces $ M\dot{K}_{p, q(\cdot)}^{\alpha(\cdot), \lambda}(\mathbb{R}^{n})$. Let $B_{k}=\{x\in\R^{n}:|x|\leq 2^{k}\}, A_{k}=\ B_{k}\setminus  B_{k-1}$ and $\chi_{_{k}}=\chi_{_{A_{k}}}$ for $k\in \mathbb{Z}$.

\begin{definition} \label{def.2.4}\ \
Suppose that $0\leq \lambda < \infty,~ 0<p< \infty$,  $q(\cdot) \in \mathscr{P}(\mathbb{R}^{n})$ and $\alpha(\cdot):\mathbb{R}^{n}\to \mathbb{R}$ with $\alpha(\cdot)\in L^{\infty}(\mathbb{R}^{n})$. The variable exponent Herz-Morrey spaces $ M\dot{K}_{p, q(\cdot)}^{\alpha(\cdot), \lambda}(\mathbb{R}^{n})$ is definded by
  $$  M\dot{K}_{p, q(\cdot)}^{\alpha(\cdot), \lambda}(\mathbb{R}^{n})=\Big\{f\in L_{\loc}^{q(\cdot)}(\mathbb{R}^{n}\backslash\{0\}):
 \|f\|_{M\dot{K}_{p, q(\cdot)}^{\alpha(\cdot), \lambda}(\mathbb{R}^{n})}<\infty \Big\},  $$
where
  $$ \|f\|_{M\dot{K}_{p, q(\cdot)}^{\alpha(\cdot), \lambda}(\mathbb{R}^{n})}=\sup_{k_{0}\in \mathbb{Z}}2^{-k_{0}\lambda}
 \Big(\sum_{k=-\infty}^{k_{0}}\|2^{k\alpha(\cdot)}f\chi_{_{\scriptstyle k}}\|_{L^{^{q(\cdot)}}(\mathbb{R}^{n})}^{p} \Big)^{\frac{1}{p}}.  $$

\end{definition}

Compare the variable Herz-Morrey space $ M\dot{K}_{p, q(\cdot)}^{\alpha(\cdot), \lambda}(\mathbb{R}^{n})$ with the variable Herz space\upcite{AD}  $\dot{K}_{q(\cdot)}^{\alpha(\cdot),p}(\R^{n})$ , where
$$\dot{K}_{q(\cdot)}^{\alpha(\cdot),p}(\R^{n})= \Big\{f\in L_{\loc}^{q(\cdot)}(\mathbb{R}^{n}\backslash\{0\}):
\sum\limits_{k=-\infty}^{\infty} \|2^{k\alpha(\cdot)} f\chi_{_{k}}\|_{L^{q(\cdot)}(\R^{n})}^{p}<\infty \Big\}.$$
Obviously, $M\dot{K}_{p, q(\cdot)}^{\alpha(\cdot), 0}(\mathbb{R}^{n})=\dot{K}_{q(\cdot)}^{\alpha(\cdot),p}(\R^{n})$. When $\alpha(\cdot)$ is constant, we have $M\dot{K}_{p, q(\cdot)}^{\alpha(\cdot), \lambda}(\mathbb{R}^{n})=M\dot{K}_{p, q(\cdot)}^{\alpha, \lambda}(\mathbb{R}^{n})$ (see \cite{W1}). If both $\alpha(\cdot)$ and $q(\cdot)$ are constant, and
$\lambda=0$, then $M\dot{K}_{p, q(\cdot)}^{\alpha(\cdot), \lambda}(\mathbb{R}^{n})=\dot{K}_{q}^{\alpha,p}(\R^{n})$ are classical Herz spaces.

\subsection{Auxiliary propositions  and lemmas }

 In this part we   state some auxiliary propositions  and lemmas which will be needed for proving  our main theorems. And we only describe partial results  we need.

 \begin{proposition}\quad \label{Pro.2.1}
Let $ q(\cdot)\in \mathscr{P}(\R^{n})$.
 \begin{list}{}{}
 \item[(I)]\ If $ q(\cdot)\in \mathscr{C}^{\log}(\R^{n})$, 
 then we have $ q(\cdot)\in \mathscr{B}(\R^{n})$.
 \item[(II)]\  $q(\cdot)\in\mathscr{B}(\R^{n})$ if and only if $q'(\cdot)\in \mathscr{B}(\R^{n})$.
\end{list}
\end{proposition}

 The  first part in Proposition \ref{Pro.2.1} is independently due to Cruz-Uribe et al\upcite{CFN} and to Nekvinda\upcite{N}  respectively. The second of Proposition \ref{Pro.2.1}  belongs to  Diening\upcite{D1}~(see Theorem 8.1 or Theorem 1.2 in \cite{CFMP}).

\begin{remark}\quad \label{rem.2}
Since
$$|q'(x)-q'(y)|\le \frac{|q(x)-q(y)|}{(q_{-}-1)^{2}},$$
 it follows at once that if $q(\cdot)\in \mathscr{C}^{\log}(\R^{n})$, then so does $q'(\cdot)$---i.e., if the condition hold, then $M$ is bounded on
  $L^{q(\cdot)} (\R^{n})$ and $L^{q'(\cdot)} (\R^{n})$.  Furthermore, Diening has proved general results on Musielak-Orlicz spaces.

\end{remark}

The order $\beta(x)$ of the fractional Hardy-type operators in Definition \ref{def.1.1} is not assumed to be continuous. We assume that it is a measurable function on $\R^n$ satisfying the following assumptions
 \begin{equation} \label{equ.2.5}
\left.
 \begin{aligned}
         \beta_{0}:= \essinf_{x\in \R^{n}} \beta(x) & >0  \\
         \esssup_{x\in \R^{n}} p(x)\beta(x) &<n           \\
         \esssup_{x\in \R^{n}} p(\infty)\beta(x) &<n
                          \end{aligned} \right\}.
                          \end{equation}


In order to prove our main results, we need the Sobolev type theorem for the space $\R^{n}$ which was proved in ref. \cite{KS} for the
exponents $p(x)$ not necessarily constant in a neigbourhood of infinity, but with some ¡®extra¡¯
power weight fixed to infinity and under the assumption that $p(x)$ takes its minimal value at
infinity.
\begin{proposition}\ \ \label{pro.2.2}
Suppose that  $p(\cdot)\in \mathscr{C}^{\log}(\R^{n}) \cap\mathscr{P}(\R^{n})$. Let
\begin{equation}\label{equ.2.6}
1<  p(\infty)\le p(x)\le p_{+}<\infty,
\end{equation}
 and  $\beta(x)$ meet condition (\ref{equ.2.5}).
Then the following weighted Sobolev-type estimate is valid for the operator $I_{\beta(\cdot)}$:
 $$\Big\|(1+|x|)^{-\gamma(x)}I_{\beta(\cdot)}(f)\Big\|_{L^{q(\cdot)}(\R^{n})}\le C\|f\|_{L^{p(\cdot)}(\R^{n})},$$
where
\begin{equation*}
\frac{1}{q(x)}=\frac{1}{p(x)}-\frac{\beta(x)}{n}
\end{equation*}
is the Sobolev exponent and
\begin{equation}\label{equ.2.7}
\gamma(x)=C_{\infty} \beta(x) \Big(1-\frac{\beta(x)}{n}\Big) \le \frac{n}{4} C_{\infty},
\end{equation}
$C_{\infty}$ being the Dini-Lipschitz constant from (\ref{equ.2.3}) which $q(\cdot)$ is replaced by $p(\cdot)$.
\end{proposition}

\begin{remark}\quad \label{rem.3}
(i)\ \ If $\beta(x)$ satisfies the condition of type (\ref{equ.2.3}): $|\beta(x)-\beta_{\infty}| \le \frac{C_{\infty}}{\ln(e+|x|)}$  $(x \in \R^{n})$,  then the weight   $(1+|x|)^{-\gamma(x)}$ is equivalent to the weight $(1+|x|)^{-\gamma_{\infty}}$.

(ii)\ \ One can also treat operator (\ref{equ.1.1}) with $\beta(x)$ replaced by $\beta(y)$. In the
case of potentials over bounded domains $\Omega$ such potentials differ unessentially, if the function  $\beta(x)$ satisfies the smoothness logarithmic condition as (\ref{equ.2.2}), since
$$C_{1}|x-y|^{n-\beta(y)}\le |x-y|^{n-\beta(x)}\le C_{2}|x-y|^{n-\beta(y)}$$
in this case ( see \cite{S}, p. 277).

(iii)\ \   When $p(\cdot)\in \mathscr{P}(\R^{n})$,  the assumption that $p(\cdot)\in \mathscr{C}^{\log}(\R^{n})$ is equivalent to assuming  $1/p(\cdot)\in \mathscr{C}^{\log}(\R^{n})$, since
$$\Big|\frac{p(x)-p(y)}{(p_{+})^{2}}\Big| \le \Big|\frac{1}{p(x)}-\frac{1}{p(y)} \Big|=\Big|\frac{p(x)-p(y)}{p(x)p(y)}\Big|\le \Big|\frac{p(x)-p(y)}{(p_{-})^{2}}\Big|.$$
And further, $1/p(\cdot)\in \mathscr{C}^{\log}(\R^{n})$ implies that $1/q(\cdot)\in \mathscr{C}^{\log}(\R^{n})$ as well.
\end{remark}

The next proposition is the generalization of variable exponents Herz spaces in \cite{AD}, and it was used in \cite{LZ}.

\begin{proposition}\ \ \label{pro.2.3}
Let  $q(\cdot)\in \mathscr{P}(\R^{n}), p\in (0,\infty)$, and $\lambda\in [0,\infty)$. If $\alpha(\cdot)\in L^{\infty}(\mathbb{R}^{n})\cap \mathscr{C}^{\log}_{0}(\R^{n})\cap \mathscr{C}^{\log}_{\infty}(\R^{n})$, then
 \begin{equation*}   
  \begin{split}
   \|f\|_{M\dot{K}_{p, q(\cdot)}^{\alpha(\cdot), \lambda}(\mathbb{R}^{n})}&=\sup_{k_{0}\in \mathbb{Z}}2^{-k_{0}\lambda} \Big(\sum_{k=-\infty}^{k_{0}}\|2^{k\alpha(\cdot)}f\chi_{_{\scriptstyle k}}\|_{L^{^{q(\cdot)}}(\mathbb{R}^{n})}^{p} \Big)^{\frac{1}{p}}  \\
   &\approx \max \bigg\{
   \sup_{k_{0}< 0 \atop k_{0}\in \mathbb{Z}} 2^{-k_{0}\lambda} \Big(\sum_{k=-\infty}^{k_{0}} 2^{k\alpha(0)p}\|f\chi_{_{\scriptstyle k}}\|_{L^{^{q(\cdot)}}(\mathbb{R}^{n})}^{p} \Big)^{\frac{1}{p}}, \\
   &{\hskip 4em} \sup_{k_{0}\ge 0 \atop k_{0}\in \mathbb{Z}} \bigg( 2^{-k_{0}\lambda} \Big(\sum_{k=-\infty}^{-1} 2^{k\alpha(0)p}\|f\chi_{_{\scriptstyle k}}\|_{L^{^{q(\cdot)}}(\mathbb{R}^{n})}^{p} \Big)^{\frac{1}{p}}           \\
    &{\hskip 6em} +  2^{-k_{0}\lambda} \Big(\sum_{k=0}^{k_{0}} 2^{k\alpha_{\infty} p}\|f\chi_{_{\scriptstyle k}}\|_{L^{^{q(\cdot)}}(\mathbb{R}^{n})}^{p} \Big)^{\frac{1}{p}}
   \bigg)
    \bigg\}.
   \end{split}
\end{equation*}

\end{proposition}

The next lemma is known as  the generalized H\"{o}lder's inequality on Lebesgue spaces with
variable exponent, and  the proof can be found in \cite{KR}.

\begin{lemma}\ (generalized H\"{o}lder's inequality) \label{Lem.1}
 \ \ Suppose that $ q(\cdot)\in \mathscr{P}(\R^{n})$, then for any $f\in L^{q(\cdot)}(\R^{n})$ and any $g\in L^{q'(\cdot)}(\R^{n})$, we have
 \begin{equation*}
 \int_{\R^{n}} |f(x)g(x)|\mathrm{d}x \le C_{q}\|f\|_{L^{q(\cdot)}(\R^{n})} \|g\|_{L^{q'(\cdot)}(\R^{n})},
\end{equation*}
where $C_{q}=1+1/q_{-}-1/q_{+}$.

\end{lemma}

The following  lemma can be found in \cite{I}.

\begin{lemma}\ \ \label{Lem.2}
Let $ q(\cdot)\in \mathscr{B}(\R^{n})$.
\begin{list}{}{}
\item[(I)]\ \ Then there exist  positive constants $\delta\in (0,1)$ and $C>0$ such that
\begin{equation*}
 \frac{\|\chi_{S}\|_{L^{q(\cdot)}(\R^{n})}} {\|\chi_{B}\|_{L^{q(\cdot)}(\R^{n})}} \le C \bigg(\frac{|S|}{|B|} \bigg)^{\delta}
\end{equation*}
 for all balls $B$ in $\R^{n}$ and all measurable subsets $S\subset B$.

 \item[(II)]\ \ Then there exists a positive constant  $C>0$ such that
 \begin{equation*}
 C^{-1}\le \frac{1}{|B|} \|\chi_{B}\|_{L^{q(\cdot)}(\R^{n})} \|\chi_{B}\|_{L^{q'(\cdot)}(\R^{n})} \le C
\end{equation*}
for all balls $B$ in $\R^{n}$.
 \end{list}
\end{lemma}

\begin{remark}\quad
(i)\ \  If $q_{_{1}}(\cdot),~q_{_{2}}(\cdot)\in \mathscr{C}^{\log}(\R^{n}) \cap \mathscr{P}(\R^{n})$, then we see that $ q'_{_{1}}(\cdot),~q_{_{2}}(\cdot) \in \mathscr{B}(\R^{n})$. Hence we can take positive constants $0<\delta_{1}<1/(q'_{_{1}})_{+},~0<\delta_{2}<1/(q_{_{2}})_{+}$  such that
\begin{equation}  \label{equ.2.9}
 \frac{\|\chi_{S}\|_{L^{q'_{1}(\cdot)}(\R^{n})}} {\|\chi_{B}\|_{L^{q'_{1}(\cdot)}(\R^{n})}} \le C \bigg(\frac{|S|}{|B|} \bigg)^{\delta_{1}},
 \ \ \  \frac{\|\chi_{S}\|_{L^{q_{2}(\cdot)}(\R^{n})}} {\|\chi_{B}\|_{L^{q_{2}(\cdot)}(\R^{n})}} \le C \bigg(\frac{|S|}{|B|} \bigg)^{\delta_{2}}
\end{equation}
 hold for all balls $B$ in $\R^{n}$ and all measurable subsets $S\subset B$ ( see \cite{WZ,I}).

(ii)\ \ On the other hand, Kopaliani\upcite{K} has proved the conclusion: If the exponent $q(\cdot)\in \mathscr{P}(\R^{n})$  equals to a constant outside some large ball, then $q(\cdot)\in \mathscr{B}(\R^{n})$ if and only if  $q(\cdot)$ satisfies the Muckenhoupt type condition
$$\sup_{Q:\hbox{cube}} \frac{1}{|Q|}\|\chi_{_{Q}}\|_{L^{q(\cdot)}(\R^{n})} \|\chi_{_{Q}}\|_{L^{q'(\cdot)}(\R^{n})} <\infty.$$
\end{remark}

\section{Main results and their proofs}

 Our main result can be stated as follows (some details see \cite{W1}).

\begin{theorem}\quad\label{thm.1}
Suppose that $q_{_{1}}(\cdot)\in \mathscr{C}^{\log}(\R^{n}) \cap\mathscr{P}(\R^{n})$ satisfies condition (\ref{equ.2.6}), and $\beta(x)$ meet condition (\ref{equ.2.5}) which $p(\cdot)$ is replaced by $q_{_{1}}(\cdot)$. Define the variable exponent $q_{_{2}}(\cdot)$ by
$$\frac{1}{q_{_{2}}(x)}=\frac{1}{q_{_{1}}(x)}-\frac{\beta(x)}{n}.$$
Let $ 0<p_{_{1}}\le {p_{_{2}}} <\infty,~ \lambda\ge 0$, and $\alpha(\cdot)\in L^{\infty}(\mathbb{R}^{n})$ be log-H\"{o}lder continuous both at the origin and at infinity,with
$ \alpha(0)\le\alpha_{\infty}< \lambda+n\delta_{1}$, where $\delta_{1}\in (0, 1/(q'_{1})_{+})$ is the constant appearing in (\ref{equ.2.9}).
Then  $$\Big\|(1+|x|)^{-\gamma(x)}\mathscr{H}_{\beta(\cdot)}(f) \Big\|_{M\dot{K}_{p_{_{2}},q_{_{2}}(\cdot)}^{\alpha(\cdot),\lambda}(\R^{n})}\le C\|f\|_{M\dot{K}_{p_{_{1}},q_{_{1}}(\cdot)}^{\alpha(\cdot),\lambda}(\R^{n})},$$
where $\gamma(x)$ is defined as in (\ref{equ.2.7}), and $C_{\infty}$ is the Dini-Lipschitz constant from (\ref{equ.2.2}) which $q_{_{1}}(\cdot)$ instead of $q(\cdot)$.

\end{theorem}

\begin{proof}\quad
For any $f\in M\dot{K}_{p_{_{1}},q_{_{1}}(\cdot)}^{\alpha(\cdot),\lambda}(\R^{n})$, if we denote   $f_j:=f\cdot\chi_{j}=f\cdot\chi_{A_j}$ for each $j\in \mathbb{Z}$, then we can write
$$f(x)=\sum_{j=-\infty}^{\infty}f(x)\cdot\chi_{j}(x) =\sum_{j=-\infty}^{\infty}f_{j}(x).
$$

By (\ref{equ.1.a}) and Lemma \ref{Lem.1}, we have
\begin{equation} \label{equ.3.1}
\begin{split}
|\mathscr{H}_{\beta(\cdot)}(f)(x)\cdot\chi_{_{\scriptstyle k}}(x)| & \le   \frac{1}{|x|^{n-\beta(x)}} \int_{B_{k}} |f(t)|\mathrm{d}t \cdot\chi_{_{\scriptstyle k}}(x)   \\
& \le C2^{^{-kn}} \sum_{j=-\infty}^{k} \|f_{j}\|_{L^{^{q_{_{1}}(\cdot)}}(\R^{n})} \|\chi_{_{\scriptstyle{j}}}\|_{L^{^{q'_{_{1}}(\cdot)}}(\R^{n})}  \cdot|x|^{\beta(x)}\chi_{_{\scriptstyle{k}}}(x).
\end{split}
\end{equation}

For Proposition \ref{pro.2.2}, we note that
\begin{equation} \label{equ.3.2}
\begin{split}
 I_{\beta(\cdot)}(\chi_{_{B_{k}}})(x)  &\ge I_{\beta(\cdot)}(\chi_{_{B_{k}}})(x)\cdot\chi_{_{B_{k}}}(x)    =   \int_{B_{k}}\frac{1}{|x-y|^{n-\beta(x)}} \mathrm{d}y  \cdot\chi_{_{B_{k}}}(x) \\
& \ge C|x|^{\beta(x)} \cdot\chi_{_{B_{k}}}(x)  \ge C|x|^{\beta(x)} \cdot\chi_{_{k}}(x).
\end{split}
\end{equation}

Using Proposition \ref{pro.2.2}, Lemma \ref{Lem.2},  (\ref{equ.2.9}),  (\ref{equ.3.1})  and (\ref{equ.3.2}), we have
\begin{equation} \label{equ.3.3}
\begin{split}
& \Big\|(1+|x|)^{-\gamma(x)}\mathscr{H}_{\beta(\cdot)}(f)\cdot\chi_{_{\scriptstyle k}}(\cdot)\Big\|_{L^{^{q_{_{2}}(\cdot)}}(\R^{n})} \\
  & \le C2^{^{-kn}} \sum_{j=-\infty}^{k} \|f_{j}\|_{L^{^{q_{_{1}}(\cdot)}}(\R^{n})} \|\chi_{_{\scriptstyle{j}}}\|_{L^{^{q'_{_{1}}(\cdot)}}(\R^{n})}  \Big\|(1+|x|)^{-\gamma(x)}I_{\beta(\cdot)}(\chi_{_{B_{k}}})\Big\|_{L^{^{q_{_{2}}(\cdot)}}(\R^{n})}   \\
& \le C2^{^{-kn}} \sum_{j=-\infty}^{k}\|f_{j}\|_{L^{^{q_{_{1}}(\cdot)}}(\R^{n})} \|\chi_{_{\scriptstyle{j}}}\|_{L^{^{q'_{_{1}}(\cdot)}}(\R^{n})}  \|\chi_{_{B_{k}}}\|_{L^{^{q_{_{1}}(\cdot)}}(\R^{n})}   \\
 & \le C \sum_{j=-\infty}^{k}\|f_{j}\|_{L^{^{q_{_{1}}(\cdot)}}(\R^{n})} \frac{\|\chi_{_{B_{j}}}\|_{L^{^{q'_{_{1}}(\cdot)}}(\R^{n})}}{\|\chi_{_{B_{k}}}\|_{L^{^{q'_{_{1}}(\cdot)}}(\R^{n})}}    \le C\sum_{j=-\infty}^{k}2^{(j-k)n\delta_{1}} \|f_{j}\|_{L^{^{q_{_{1}}(\cdot)}}(\R^{n})}.
\end{split}
\end{equation}

Because of $0<p_{_{1}}/p_{_{2}}\le 1$,  applying  inequality
\begin{equation}     \label{equ.3.4}
\bigg(\sum_{i=-\infty}^{\infty}|a_{i}|\bigg)^{ p_{_{1}}/p_{_{2}}} \le \sum_{i=-\infty}^{\infty} |a_{i}|^{ p_{_{1}}/ p_{_{2}}},
\end{equation}
and   Proposition \ref{pro.2.3}, then we have
\begin{eqnarray*}
&&\; \Big\|(1+|x|)^{-\gamma(x)}\mathscr{H}_{\beta(\cdot)}(f) \Big\|^{p_{_{1}}}_{M\dot{K}_{p_{_{2}},q_{_{2}}(\cdot)}^{\alpha(\cdot),\lambda}(\R^{n})} \\
 &&\; \le \max \bigg\{
   \sup_{k_{0}< 0 \atop k_{0}\in \mathbb{Z}} 2^{-k_{0}\lambda p_{_{1}}} \Big(\sum_{k=-\infty}^{k_{0}} 2^{k\alpha(0){p_{_{1}}}}\|(1+|x|)^{-\gamma(x)}\mathscr{H}_{\beta(\cdot)}(f)\cdot\chi_{_{\scriptstyle k}}\|_{L^{^{q_{_{2}}(\cdot)}}(\mathbb{R}^{n})}^{p_{_{1}}}\Big),       \\
&&\;{\hskip 4em} \sup_{k_{0}\ge 0 \atop k_{0}\in \mathbb{Z}} \bigg( 2^{-k_{0}\lambda p_{_{1}}} \Big(\sum_{k=-\infty}^{-1} 2^{k\alpha(0){p_{_{1}}}} \|(1+|x|)^{-\gamma(x)}\mathscr{H}_{\beta(\cdot)}(f)\cdot\chi_{_{\scriptstyle k}}\|_{L^{^{q_{_{2}}(\cdot)}}(\mathbb{R}^{n})}^{p_{_{1}}}\Big)           \\
&&\;{\hskip 6em} +  2^{-k_{0}\lambda p_{_{1}}} \Big(\sum_{k=0}^{k_{0}} 2^{k\alpha_{\infty}  p_{_{1}}} \|(1+|x|)^{-\gamma(x)}\mathscr{H}_{\beta(\cdot)}(f)\cdot\chi_{_{\scriptstyle k}}\|_{L^{^{q_{_{2}}(\cdot)}}(\mathbb{R}^{n})}^{p_{_{1}}}\Big)
   \bigg)
    \bigg\}    \\
 &&\;=\max\{E_{1},E_{2}+E_{3}\},
\end{eqnarray*}
where
\begin{eqnarray*}
E_{1} &=& \sup_{k_{0}< 0 \atop k_{0}\in \mathbb{Z}} 2^{-k_{0}\lambda p_{_{1}}} \Big(\sum_{k=-\infty}^{k_{0}} 2^{k\alpha(0){p_{_{1}}}}\|(1+|x|)^{-\gamma(x)}\mathscr{H}_{\beta(\cdot)}(f)\cdot\chi_{_{\scriptstyle k}}\|_{L^{^{q_{_{2}}(\cdot)}}(\mathbb{R}^{n})}^{p_{_{1}}}\Big),  \\
 E_{2} &=& \sup_{k_{0}\ge 0 \atop k_{0}\in \mathbb{Z}}   2^{-k_{0}\lambda p_{_{1}}} \Big(\sum_{k=-\infty}^{-1} 2^{k\alpha(0){p_{_{1}}}} \|(1+|x|)^{-\gamma(x)}\mathscr{H}_{\beta(\cdot)}(f)\cdot\chi_{_{\scriptstyle k}}\|_{L^{^{q_{_{2}}(\cdot)}}(\mathbb{R}^{n})}^{p_{_{1}}}\Big),           \\
 E_{3} &=& \sup_{k_{0}\ge 0 \atop k_{0}\in \mathbb{Z}} 2^{-k_{0}\lambda p_{_{1}}} \Big(\sum_{k=0}^{k_{0}} 2^{k\alpha_{\infty}  p_{_{1}}} \|(1+|x|)^{-\gamma(x)}\mathscr{H}_{\beta(\cdot)}(f)\cdot\chi_{_{\scriptstyle k}}\|_{L^{^{q_{_{2}}(\cdot)}}(\mathbb{R}^{n})}^{p_{_{1}}}\Big).
\end{eqnarray*}

To estimate $E_{1},E_{2}$ and $E_{3}$, we need the following fact. By the condition of $\alpha(\cdot)$  and  Proposition \ref{pro.2.3}, we have

Case 1 ($j<0$),
\begin{eqnarray}
\begin{split} \label{equ.3.5}
\|f_{j}\|_{L^{^{q_{_{1}}(\cdot)}}(\R^{n})} &= 2^{-j\alpha(0)}\Big(2^{j\alpha(0)p_{_1}} \|f_{j}\|^{p_{_1}}_{L^{^{q_{_{1}}(\cdot)}}(\R^{n})}\Big)^{1/p_{_1}}\\
&\le  2^{-j\alpha(0)}\bigg(\sum_{i=-\infty}^j 2^{i\alpha(0)p_{_1}} \|f_{i}\|^{p_{_1}}_{L^{^{q_{_{1}}(\cdot)}}(\R^{n})}\bigg)^{1/p_{_1}}\\
&\le  2^{j(\lambda-\alpha(0))}\bigg(2^{-j\lambda} \Big(\sum_{i=-\infty}^j   \|2^{i\alpha(\cdot)}f_{i}\|^{p_{_1}}_{L^{^{q_{_{1}}(\cdot)}} (\R^{n})}\Big)^{1/p_{_1}}\bigg)\\
&\le  C 2^{j(\lambda-\alpha(0))} \|f\|_{M\dot{K}_{p_{_{1}},q_{_{1}}(\cdot)}^{\alpha(\cdot),\lambda}(\R^{n})}.
 \end{split}
\end{eqnarray}

Case 2 ($j\ge 0$),
\begin{eqnarray}
\begin{split} \label{equ.3.6}
\|f_{j}\|_{L^{^{q_{_{1}}(\cdot)}}(\R^{n})} &= 2^{-j\alpha_{\infty}}\Big(2^{j\alpha_{\infty} p_{_1}} \|f_{j}\|^{p_{_1}}_{L^{^{q_{_{1}}(\cdot)}}(\R^{n})}\Big)^{1/p_{_1}}\\
&\le  2^{-j\alpha_{\infty}}\bigg(\sum_{i=0}^j 2^{i\alpha_{\infty}p_{_1}} \|f_{i}\|^{p_{_1}}_{L^{^{q_{_{1}}(\cdot)}}(\R^{n})}\bigg)^{1/p_{_1}}\\
&\le  2^{j(\lambda-\alpha_{\infty})}\bigg(2^{-j\lambda} \Big(\sum_{i=-\infty}^j   \|2^{i\alpha(\cdot)}f_{i}\|^{p_{_1}}_{L^{^{q_{_{1}}(\cdot)}} (\R^{n})}\Big)^{1/p_{_1}}\bigg)\\
&\le  C 2^{j(\lambda-\alpha_{\infty})} \|f\|_{M\dot{K}_{p_{_{1}},q_{_{1}}(\cdot)}^{\alpha(\cdot),\lambda}(\R^{n})}.
 \end{split}
\end{eqnarray}

For $E_{1}$, note that $j<0$, combining (\ref{equ.3.3}) and (\ref{equ.3.5}), and using $ \alpha(0)\le \alpha_{\infty}< \lambda+n\delta_{1}$, it follows that
\begin{eqnarray*}
E_{1}
&\le& C \sup_{k_{0}< 0 \atop k_{0}\in \mathbb{Z}}  2^{-k_{0}\lambda p_{_{1}}} \bigg(\sum_{k=-\infty}^{k_{0}} 2^{k\alpha(0){p_{_{1}}}} \Big(\sum_{j=-\infty}^{k} 2^{(j-k)n\delta_{1}} \|f_{j}\|_{L^{^{q_{_{1}}(\cdot)}}(\R^{n})} \Big)^{p_{_{1}}}\bigg)  \\
&\le& C \|f\|^{p_{1}}_{M\dot{K}_{p_{_{1}},q_{_{1}}(\cdot)}^{\alpha(\cdot),\lambda}(\R^{n})} \sup_{k_{0}< 0 \atop k_{0}\in \mathbb{Z}}  2^{-k_{0}\lambda p_{_{1}}} \bigg(\sum_{k=-\infty}^{k_{0}} 2^{k\alpha(0){p_{_{1}}}} \Big(\sum_{j=-\infty}^{k} 2^{(j-k)n\delta_{1}} 2^{j(\lambda-\alpha(0))} \Big)^{p_{_{1}}}\bigg)  \\
&\le& C \|f\|^{p_{1}}_{M\dot{K}_{p_{_{1}},q_{_{1}}(\cdot)}^{\alpha(\cdot),\lambda}(\R^{n})} \sup_{k_{0}< 0 \atop k_{0}\in \mathbb{Z}}  2^{-k_{0}\lambda p_{_{1}}} \bigg(\sum_{k=-\infty}^{k_{0}} 2^{k\lambda p_{_{1}}} \Big(\sum_{j=-\infty}^{k} 2^{(j-k)(n\delta_{1}+\lambda-\alpha(0))} \Big)^{p_{_{1}}}\bigg)  \\
&\le& C \|f\|^{p_{1}}_{M\dot{K}_{p_{_{1}},q_{_{1}}(\cdot)}^{\alpha(\cdot),\lambda}(\R^{n})} \sup_{k_{0}< 0 \atop k_{0}\in \mathbb{Z}}  2^{-k_{0}\lambda p_{_{1}}} \Big(\sum_{k=-\infty}^{k_{0}} 2^{k\lambda p_{_{1}}} \Big)   \le C \|f\|^{p_{1}}_{M\dot{K}_{p_{_{1}},q_{_{1}}(\cdot)}^{\alpha(\cdot),\lambda}(\R^{n})}.
\end{eqnarray*}

We omit the estimate of $E_{2}$ since it is essentially similar to that of $E_{1}$.

Now we only simply estimate $E_{3}$, note that $j\ge 0$, combining (\ref{equ.3.3}) and (\ref{equ.3.6}), and using $ \alpha(0)\le \alpha_{\infty}< \lambda+n\delta_{1}$, we have
 \begin{eqnarray*}
E_{3}
&\le& C \sup_{k_{0}\ge 0 \atop k_{0}\in \mathbb{Z}}  2^{-k_{0}\lambda p_{_{1}}} \bigg(\sum_{k=0}^{k_{0}} 2^{k\alpha_{\infty}{p_{_{1}}}} \Big(\sum_{j=-\infty}^{k} 2^{(j-k)n\delta_{1}} \|f_{j}\|_{L^{^{q_{_{1}}(\cdot)}}(\R^{n})} \Big)^{p_{_{1}}}\bigg)  \\
&\le& C \|f\|^{p_{1}}_{M\dot{K}_{p_{_{1}},q_{_{1}}(\cdot)}^{\alpha(\cdot),\lambda}(\R^{n})} \sup_{k_{0}\ge 0 \atop k_{0}\in \mathbb{Z}}  2^{-k_{0}\lambda p_{_{1}}} \bigg(\sum_{k=0}^{k_{0}} 2^{k\lambda p_{_{1}}} \Big(\sum_{j=-\infty}^{k} 2^{(j-k)(n\delta_{1}+\lambda-\alpha_{\infty})} \Big)^{p_{_{1}}}\bigg)  \\
&\le& C \|f\|^{p_{1}}_{M\dot{K}_{p_{_{1}},q_{_{1}}(\cdot)}^{\alpha(\cdot),\lambda}(\R^{n})} \sup_{k_{0}\ge 0 \atop k_{0}\in \mathbb{Z}}  2^{-k_{0}\lambda p_{_{1}}} \Big(\sum_{k=0}^{k_{0}} 2^{k\lambda p_{_{1}}} \Big)   \le C \|f\|^{p_{1}}_{M\dot{K}_{p_{_{1}},q_{_{1}}(\cdot)}^{\alpha(\cdot),\lambda}(\R^{n})}.
\end{eqnarray*}

Combining all the estimates for $E_{i}~(i=1,2,3)$ together,  the proof of Theorem \ref{thm.1} is completed.
\end{proof}


\begin{theorem}\quad\label{thm.2}
Let $\lambda,p_{_{1}}, {p_{_{2}}}, q_{_{1}}(\cdot), q_{_{2}}(\cdot),  \beta(x), C_{\infty} $  be as in Theorem \ref{thm.1}. Suppose that $\alpha\in L^{\infty}(\mathbb{R}^{n})$ be log-H\"{o}lder continuous both at the origin and at infinity, and
$ \lambda-n\delta_{2}<\alpha(0)\le\alpha_{\infty}$,    where $\delta_{2}\in (0, 1/(q_{2})_{+})$ is the constant appearing in (\ref{equ.2.9}).
Then  $$\Big\|(1+|x|)^{-\gamma(x)}\mathscr{H}^{\ast}_{\beta(\cdot)}(f) \Big\|_{M\dot{K}_{p_{_{2}},q_{_{2}}(\cdot)}^{\alpha(\cdot),\lambda}(\R^{n})}\le C\|f\|_{M\dot{K}_{p_{_{1}},q_{_{1}}(\cdot)}^{\alpha(\cdot),\lambda}(\R^{n})}.$$

\end{theorem}

\begin{proof}\quad Similar to the proof of Theorem \ref{thm.1},  therefore, we only give a simple proof.

For simplicity, for any $f\in M\dot{K}_{p_{_{1}},q_{_{1}}(\cdot)}^{\alpha(\cdot),\lambda}(\R^{n})$, we  write
$$f(x)=\sum_{j=-\infty}^{\infty}f(x)\cdot\chi_{j}(x) =\sum_{j=-\infty}^{\infty}f_{j}(x).
$$

By (\ref{equ.1.b}) and Lemma \ref{Lem.1}, we have
\begin{equation} \label{equ.3.7}
\begin{split}
&\Big|(1+|x|)^{-\gamma(x)}\mathscr{H}^{\ast}_{\beta(\cdot)}(f)(x)\cdot\chi_{_{\scriptstyle k}}(x)\Big|
 \le C\int_{\R^{n}\setminus B_{k}} |f(t)| |x|^{\beta(x)-n} \mathrm{d}t \cdot (1+|x|)^{-\gamma(x)}\chi_{_{\scriptstyle k}}(x)   \\
& \le C\sum_{j=k+1}^{\infty}  \|f_{j}\|_{L^{^{q_{_{1}}(\cdot)}}(\R^{n})} \Big\|(1+|x|)^{-\gamma(x)} |\cdot|^{\beta(x)-n}\chi_{_{\scriptstyle{j}}} (\cdot) \Big\|_{L^{^{q'_{_{1}}(\cdot)}}(\R^{n})}  \cdot\chi_{_{\scriptstyle{k}}}(x).
\end{split}
\end{equation}
Similar to (\ref{equ.3.2}), we give
\begin{equation} \label{equ.3.8}
\begin{split}
 I_{\beta(\cdot)}(\chi_{_{B_{j}}})(x)  &\ge I_{\beta(\cdot)}(\chi_{_{B_{j}}})(x)\cdot\chi_{_{B_{j}}}(x)     \ge C|x|^{\beta(x)} \cdot\chi_{_{j}}(x).
\end{split}
\end{equation}

Applying Proposition \ref{pro.2.2},  Lemma \ref{Lem.2}, (\ref{equ.2.9}),  (\ref{equ.3.7})  and (\ref{equ.3.8}), we obtain
\begin{equation} \label{equ.3.9}
\begin{split}
& \Big\|(1+|x|)^{-\gamma(x)}\mathscr{H}^{\ast}_{\beta(\cdot)}(f)\cdot\chi_{_{\scriptstyle k}} (\cdot) \Big\|_{L^{^{q_{_{2}}(\cdot)}}(\R^{n})} \\
  & \le C\sum_{j=k+1}^{\infty}   \|f_{j}\|_{L^{^{q_{_{1}}(\cdot)}}(\R^{n})} \|\chi_{_{\scriptstyle{k}}}\|_{L^{^{q_{_{2}}(\cdot)}}(\R^{n})} \cdot 2^{^{-jn}} \Big\|(1+|x|)^{-\gamma(x)} I_{\beta(\cdot)}(\chi_{_{B_{j}}})\Big\|_{L^{^{q'_{_{1}}(\cdot)}}(\R^{n})}   \\
& \le C \sum_{j=k+1}^{\infty}\|f_{j}\|_{L^{^{q_{_{1}}(\cdot)}}(\R^{n})} \frac{\|\chi_{_{B_{k}}}\|_{L^{^{q_{_{2}}(\cdot)}}(\R^{n})}}{\|\chi_{_{B_{j}}}\|_{L^{^{q_{_{2}}(\cdot)}}(\R^{n})}}    \le C\sum_{j=k+1}^{\infty} 2^{(k-j)n\delta_{2}} \|f_{j}\|_{L^{^{q_{_{1}}(\cdot)}}(\R^{n})}.
\end{split}
\end{equation}

By (\ref{equ.3.4}) and Proposition \ref{pro.2.3},  we have
\begin{eqnarray*}
&&\; \Big\|(1+|x|)^{-\gamma(x)}\mathscr{H}^{\ast}_{\beta(\cdot)}(f) \Big\|^{p_{_{1}}}_{M\dot{K}_{p_{_{2}},q_{_{2}}(\cdot)}^{\alpha(\cdot),\lambda}(\R^{n})} \le \max\{E_{1},E_{2}+E_{3}\},
\end{eqnarray*}
where
\begin{eqnarray*}
E_{1} &=& \sup_{k_{0}< 0 \atop k_{0}\in \mathbb{Z}} 2^{-k_{0}\lambda p_{_{1}}} \Big(\sum_{k=-\infty}^{k_{0}} 2^{k\alpha(0){p_{_{1}}}}\|(1+|x|)^{-\gamma(x)}\mathscr{H}^{\ast}_{\beta(\cdot)}(f)\cdot\chi_{_{\scriptstyle k}}\|_{L^{^{q_{_{2}}(\cdot)}}(\mathbb{R}^{n})}^{p_{_{1}}}\Big),  \\
 E_{2} &=& \sup_{k_{0}\ge 0 \atop k_{0}\in \mathbb{Z}}   2^{-k_{0}\lambda p_{_{1}}} \Big(\sum_{k=-\infty}^{-1} 2^{k\alpha(0){p_{_{1}}}} \|(1+|x|)^{-\gamma(x)}\mathscr{H}^{\ast}_{\beta(\cdot)}(f)\cdot\chi_{_{\scriptstyle k}}\|_{L^{^{q_{_{2}}(\cdot)}}(\mathbb{R}^{n})}^{p_{_{1}}}\Big),           \\
 E_{3} &=& \sup_{k_{0}\ge 0 \atop k_{0}\in \mathbb{Z}} 2^{-k_{0}\lambda p_{_{1}}} \Big(\sum_{k=0}^{k_{0}} 2^{k\alpha_{\infty}  p_{_{1}}} \|(1+|x|)^{-\gamma(x)}\mathscr{H}^{\ast}_{\beta(\cdot)}(f)\cdot\chi_{_{\scriptstyle k}}\|_{L^{^{q_{_{2}}(\cdot)}}(\mathbb{R}^{n})}^{p_{_{1}}}\Big).
\end{eqnarray*}

For $E_{1}, E_{2}$ and $E_{3}$,   combining  (\ref{equ.3.5}), (\ref{equ.3.6}) and (\ref{equ.3.9}), and using $ \lambda-n\delta_{2}<\alpha(0)\le\alpha_{\infty}$, we have
$$
 E_{i}    \le C \|f\|^{p_{1}}_{M\dot{K}_{p_{_{1}},q_{_{1}}(\cdot)}^{\alpha(\cdot),\lambda}(\R^{n})},  (i=1,2, 3).
 $$

Combining all the estimates for $E_{i}~(i=1,2,3)$ together,  the proof of Theorem \ref{thm.2} is completed.
 \end{proof}


In particular, when $\gamma(x)=0$, $\alpha(\cdot)$ and $\beta(\cdot)$ are constant exponent, the  main results above are proved by Zhang and Wu in \cite{ZW}. Let  $\alpha(\cdot)$ be constant exponent, then the above results can be founded in \cite{W1}.  And when $\lambda=0$, our main results are also valid.

\begin{acknowledgments}
 The authors cordially  thank the anonymous referees who gave valuable  suggestions and useful comments which have lead to the improvement of this paper.  Meanwhile, the authors also would like to thank  the partial support from  National Natural Science Foundation of China (Grant No.11571160) and   Mudanjiang Normal University (No.QY2014007).
\end{acknowledgments}


\bigskip

\end{document}